\documentclass[10pt,a4paper,leqno]{article}
\usepackage{fullpage}
\usepackage{amsfonts,amsmath,amssymb}
\usepackage{amsthm}
\usepackage{graphicx}
\usepackage{subfigure}
\usepackage{sectsty}
\usepackage{authblk}
\theoremstyle{plain}
\newtheorem{theorem}{Theorem}[section]
\newtheorem{lemma}[theorem]{Lemma}

\theoremstyle{definition}

\theoremstyle{remark}

\sectionfont{\large}

\title{Inverse spectral problem of Sturm-Liouville equation with many  frozen arguments}
\date{}
\author[1]{\small Chung-Tsun Shieh}
\author[2]{\small Tzong-Mo Tsai}
\affil[1]{\footnotesize Department of Mathematics, Tamkang University, New Taipei City, 25137, Taiwan\\ctshieh@mail.tku.edu.tw.}
\affil[2]{\footnotesize General Education Center, Ming Chi University of Technology, New Taipei City, 24301, Taiwan\\tsaitm@mail.mcut.edu.tw}

\begin{document}
\maketitle
\begin{abstract}
This research  was devoted to  investigate the inverse spectral problem of Sturm-Liouville operator with many frozen arguments. Under some assumptions, the authors obtained  uniqueness theorems. At the end, a numerical simulation for the  inverse problem was presented.
\end{abstract}

\\

\section{Introduction}
In this paper, the authors consider the boundary value problem $L(q(x),F_n, \alpha, \beta)$  which consists of a nonlocal Sturm-Liouville equation
\begin{equation}\label{E1}
-y''(x)+q(x)\sum_{i=1}^ny(a_i)=\lambda y(x),\; 0<x<\pi,
\end{equation}
associated with the boundary conditions 
\begin{equation}\label{E2}
	y^{(\alpha)}(0)=0=y^{(\beta)}(\pi),
\end{equation}
where $\lambda$ is the spectral parameter,  $F_{n}=\{a_i\}_{i=1}^{n}$, $0<a_i<a_j<\pi$ if $i<j,$ is the set of frozen arguments, $\alpha,\; \beta \in \{0,1\}$ and $q(x)\in L^2(0,1).$ Since the analyses for the problems of all cases of $\{\alpha,\beta\}$ which we concerned are similar, we shall only treat the case $\alpha=0=\beta.$ 

As applications, equation \eqref{E1} occurs in the class of the so-called loaded equations\cite{Nakh82, Nakh12}, which contain the values of the unknown function at some fixed points. Such equations have been applied for studying grounder dynamics \cite{Nakh82, NakhBor77}, heating processes \cite{Isk, Nakh76}, feedback-like phenomena in vibrations of a wire affected by a magnetic field \cite{Kral, BH21}.

Let $\{\lambda_m\}_{m\ge1}$ be the eigenvalues of $L(q(x),F_{n},\alpha,\beta)$.  In this paper, the authors studied the following inverse problem:

\medskip
{\bf Inverse Problem 1.} Given $\{\lambda_m\}_{m\ge1},$ $F_{n},\;\alpha$ and $\beta,$ determine the potential $q(x).$

\medskip
Inverse Problem 1 of the case  $n=1$ had been studied by many researchers \cite{BBV, BV, BK, Wang20, KuzIrr, Kuz22}. In \cite{BBV,
BV, BK, TLBCS21}, diverse cases of the triple $(a_1,\alpha,\beta)$ with rational $a_1=j/k,\; \gcd(k,j)=1, $ were considered. In particular, it was established that the
inverse problem may be uniquely solvable or not, depending on the parameters $\alpha,\,\beta$ and also on  $(k,j).$ 
Concerning irrational values of $a_1,$ it was established by Y. P. Wang et al in \cite{Wang20} that  all pairs
$(\alpha,\beta),$  the solution of Inverse Problem 1 for $n=1$ is always uniquely  determined by  the spectrum set.
Recently  Maria Kuznetsova \cite{KuzIrr, Kuz22}  obtained a unified asymptotic formula covering
both irrational and rational cases.

The paper is organized as follows. In the next section, the uniqueness theorems were established. In section 3, numerical algorithm and simulations  were presented.

\section{Main Results}
In this section , we shall present our approach which is different from that in \cite{BBV, BV, BK} and \cite{Wang20}. We shall write $\lambda=\rho^2.$
\begin{lemma}
	Denote $S(x,\lambda)$ the solution of \eqref{E1} associated with the initial condition 
	$$ y(0,\lambda)=0,\;  y'(0,\lambda)=1.$$
	Then $S(x,\lambda)$ satisfies
	\begin{equation}\label{E3}
	S(x,\lambda)=\dfrac{\sin\rho x}{\rho}+\int_0^x \dfrac{\sin\rho(x-t)}{\rho}q(t)\left(\sum_{i=1}^n S(a_i,\lambda)\right)\, dt
	\end{equation}
\end{lemma}
	The derivation is not complicated, we shall omit the details.

Next, we shall start with the simplest case $n=1$ to show \eqref{E3} works for our study.  For $n=1,$  
\eqref{E3} turns to 
\begin{equation}
S(x,\lambda)=\dfrac{\sin\rho x}{\rho}+S(a_1,\lambda)\int_0^x \dfrac{\sin\rho(x-t)}{\rho}q(t)\, dt.
\end{equation}
Hence we have 
\begin{equation}
	\begin{cases}
		S(a_1,\lambda)&=\dfrac{\sin\rho a_1}{\rho}+S(a_1,\lambda)\int_0^{a_1} \dfrac{\sin\rho(a_1-t)}{\rho}q(t) \, dt,\\
		S(\pi,\lambda)&=\dfrac{\sin\rho \pi }{\rho}+S(a_1,\lambda)\int_0^{\pi}\dfrac{\sin\rho (\pi-t)}{\rho}q(t) \, dt,\\
		\end{cases}
\end{equation}
that is 
\begin{equation}
\begin{bmatrix}\dfrac{\sin\rho a_1}{\rho} & \int_0^{a_1}q(t)\dfrac{\sin \rho(a_1-t)}{\rho}\, dt-1 \\
\dfrac{\sin\rho \pi}{\rho} & \int_0^{\pi}q(t)\dfrac{\sin \rho(\pi-t)}{\rho}\, dt
	\end{bmatrix}\begin{bmatrix} 1 \\ S(a_1,\lambda)\end{bmatrix}=\begin{bmatrix} 0 \\ S(\pi,\lambda)\end{bmatrix}.
\end{equation}
Thus $\lambda$ is an eigenvalue of $L(q(x),F_1, 0, 0)$ if and only if $\lambda$ is a zero of the characteristic function 
\begin{equation}\label{E7}
\triangle_1(\lambda)=\det\left ( \begin{bmatrix}\dfrac{\sin\rho a_1}{\rho} & \int_0^{a_1}q(t)\dfrac{\sin \rho(a_1-t)}{\rho}\, dt-1 \\
\dfrac{\sin\rho \pi}{\rho} & \int_0^{\pi}q(t)\dfrac{\sin \rho(\pi-t)}{\rho}\, dt
	\end{bmatrix}\right).	
\end{equation}
One can easily see that 
\begin{equation}\label{E8}
\triangle_1(\lambda)=\dfrac{\sin\rho\pi}{\rho}	+\dfrac{\sin\rho a_1}{\rho}\int_0^\pi q(t)\dfrac{\sin\rho(\pi-t)}{\rho}\,dt-\dfrac{\sin\rho \pi}{\rho}\int_0^{a_1}q(t)\dfrac{\sin \rho(a_1-t)}{\rho}\, dt
\end{equation}
After some calculation, we  have 
\begin{equation}\label{E9}
\triangle_1(\lambda)=\dfrac{\sin\rho\pi}{\rho}+\dfrac{\sin\rho a_1}{\rho^2}\int_{a_1}^\pi q(t)\sin\rho(\pi-t)\, dt +\dfrac{\sin\rho(\pi-a_1)}{\rho^2}\int_0^{a_1}q(t)\sin\rho t\, dt.
\end{equation}
The right side of \eqref{E9} is the form of characteristic function which appeared in \cite{Wang20}. 

To go one step further, we shall consider the case for $n=2.$  Then
   \begin{equation}
   	S(x,\lambda)=\dfrac{\sin\rho x}{\rho}+(S(a_1,\lambda)+S(a_2,\lambda))\int_0^x q(t)\dfrac{\sin\rho(x-t)}{\rho}\, dt.
   \end{equation} 
   Hence we have 
   \begin{equation}\label{E11}
   \begin{cases}
   S(a_1,\lambda)&=\dfrac{\sin\rho a_1}{\rho}	+(S(a_1,\lambda)+S(a_2,\lambda))\int_0^{a_1} q(t)\dfrac{\sin\rho(a_1-t)}{\rho}\, dt,\\
S(a_2,\lambda)&=\dfrac{\sin\rho a_2}{\rho}	+(S(a_1,\lambda)+S(a_2,\lambda))\int_0^{a_2} q(t)\dfrac{\sin\rho(a_2-t)}{\rho}\, dt,\\
S(\pi,\lambda)&=\dfrac{\sin\rho \pi}{\rho}	+(S(a_1,\lambda)+S(a_2,\lambda))\int_0^{\pi} q(t)\dfrac{\sin\rho(\pi -t)}{\rho}\, dt.
   \end{cases}
	\end{equation}
\eqref{E11} leads to 
\begin{equation*}
	\begin{bmatrix}
	\dfrac{\sin\rho a_1}{\rho} &\int_0^{a_1}q(t)\dfrac{\sin\rho(a_1-t)}{\rho}\, dt -1 & \int_0^{a_1}q(t) \dfrac{\sin \rho (a_1-t)}{\rho}\, dt \\
	\dfrac{\sin\rho a_2}{\rho} &\int_0^{a_2}q(t) \dfrac{\sin \rho (a_2-t)}{\rho}\, dt  &\int_0^{a_2}q(t)\dfrac{\sin\rho(a_2-t)}{\rho}\, dt -1\\
	\dfrac{\sin\rho \pi}{\rho}&\int_0^{\pi}q(t) \dfrac{\sin \rho (\pi-t)}{\rho}\, dt & \int_0^{\pi}q(t) \dfrac{\sin \rho (\pi-t)}{\rho}\, dt
	\end{bmatrix}\begin{bmatrix} 1 \\ S(a_1,\lambda) \\ S(a_2,\lambda)\end{bmatrix}=\begin{bmatrix} 0 \\ 0 \\ S(\pi,\lambda)\end{bmatrix}.
\end{equation*}

Thus $\lambda$ is an eigenvalue of $L(q(x),F_2, 0, 0)$ if and only if $\lambda$ is a zero of the characteristic function
\begin{equation}\label{E12}
\triangle_2(\lambda)=\det \left(\begin{bmatrix}
	\dfrac{\sin\rho a_1}{\rho} &\int_0^{a_1}q(t)\dfrac{\sin\rho(a_1-t)}{\rho}\, dt -1 & \int_0^{a_1}q(t) \dfrac{\sin \rho (a_1-t)}{\rho}\, dt \\
	\dfrac{\sin\rho a_2}{\rho} &\int_0^{a_2}q(t) \dfrac{\sin \rho (a_2-t)}{\rho}\, dt  &\int_0^{a_2}q(t)\dfrac{\sin\rho(a_2-t)}{\rho}\, dt -1\\
	\dfrac{\sin\rho \pi}{\rho}&\int_0^{\pi}q(t) \dfrac{\sin \rho (\pi-t)}{\rho}\, dt & \int_0^{\pi}q(t) \dfrac{\sin \rho (\pi-t)}{\rho}\, dt
	\end{bmatrix}\right).
\end{equation}
Note that one achieves 
\begin{align}
\triangle_2(\lambda)&=\begin{vmatrix} \dfrac{\sin\rho a_1}{\rho} &\int_0^{a_1}q(t)\dfrac{\sin\rho(a_1-t)}{\rho}\, dt -1 & 1 \\
	\dfrac{\sin\rho a_2}{\rho} &\int_0^{a_2}q(t) \dfrac{\sin \rho (a_2-t)}{\rho}\, dt  & -1\\
	\dfrac{\sin\rho \pi}{\rho}&\int_0^{\pi}q(t) \dfrac{\sin \rho (\pi-t)}{\rho}\, dt & 0 \end{vmatrix}	\label{E13} \\
	&=\dfrac{\sin \rho \pi}{\rho}+\dfrac{\sin\rho a_1}{\rho^2}\int_0^\pi q(t)\sin\rho(\pi-t)\, dt-\dfrac{\sin\rho \pi }{\rho^2}\int_0^{a_1}q(t)\sin\rho(a_1-t)\, dt \notag \\  & \; +\dfrac{\sin\rho a_2}{\rho^2}\int_0^\pi q(t)\sin\rho(\pi-t)\, dt -\dfrac{\sin\rho \pi}{\rho^2}\int_0^{a_2}q(t)\sin\rho(a_2-t)\,dt\notag 
\end{align}
after a simplification of \eqref{E12}.
Denote $\Lambda(q(x),F_2)=\{\lambda_m\}_{m=1}^\infty$ the eigenvalues of $L(q(x),F_2,0,0),$ where $F_2=\{a_1,a_2\}.$ By Hadamard's factorization theorem, we have 
\begin{equation}\label{E14}	
\triangle_2(\lambda)=\pi\prod_{n=1}^\infty \dfrac{(\lambda_n-\lambda)}{n^2}.\end{equation}

From the analysis above, we have  following uniqueness theorem 
\begin{theorem}\label{T1}
Let $F_2=\{a_1,a_2\},$ where $0< a_1<a_2<\pi.$
Define $G_2(\rho)=\sin(\rho a_1)+\sin(\rho a_2),$ and denote $Z(G_2)$ the set of zeros of $G_2(\rho).$ Suppose $Z(G_2)\cap \mathbb N=\emptyset $ and 
 $\Lambda(q(x),F_2)=\Lambda(\tilde q(x),F_2).$ Then $q(x)=\tilde q(x)$ almost everywhere in $(0,\pi).$  
\end{theorem}

Note if $(a_1\pm a_2)/\pi$ are irrational, then $Z(G_2)\cap \mathbb N=\emptyset. $

\begin{proof}
Denote $\hat q(x)=q(x)-\tilde q(x),$  $\triangle_2(\lambda)$ the characteristic function of $L(q(x),F_2,0,0)$ and $\tilde \triangle_2(\lambda)$ the characteristic function of $L(\tilde q(x),F_2,0,0).$ According to \eqref{E13} and \eqref{E14}, one can obtain
\begin{align}
	\triangle_2(\lambda)-\tilde\triangle_2(\lambda)=0=\begin{vmatrix} \dfrac{\sin\rho a_1}{\rho} &\int_0^{a_1}\hat q(t)\dfrac{\sin\rho(a_1-t)}{\rho}\, dt  & 1 \\
	\dfrac{\sin\rho a_2}{\rho} &\int_0^{a_2}\hat q(t) \dfrac{\sin \rho (a_2-t)}{\rho}\, dt  & -1\\
	\dfrac{\sin\rho \pi}{\rho}&\int_0^{\pi}\hat q(t) \dfrac{\sin \rho (\pi-t)}{\rho}\, dt & 0 \end{vmatrix}	
	\end{align}
The assumption $Z(G_2)\cap \mathbb N=\emptyset $ ensure   $(\dfrac{\sin m a_1}{m}, \dfrac{\sin m a_2}{m},\dfrac{\sin m \pi}{m})^{\text{t}}$ and $(1,-1,0)^{\text{t}}$ are independent,
we have $$c_m\int_0^\pi \hat q(t) \sin m(\pi-t)\, dt=0 $$
for some $c_m\ne 0$  and for all $m\in \mathbb N.$  We have $\hat q(t)=0$ almost everywhere in $(0,\pi).$ This completes the proof.
 \end{proof}

Apply the same arguments, Theorem \ref{T1} can be generalized as following 
\begin{theorem}\label{T2}
	Let $F_n=\{a_1,a_2,\dots, a_n\},$ where $0< a_1<a_2	<\cdots<a_n<\pi.$ Define $G_n(\rho)=\sum_{i=1}^n \sin(\rho a_i).$
Suppose the following conditions 
\begin{itemize}
	\item[(i)] $Z(G_n)\cap \mathbb N=\emptyset,$
	\item[(ii)] $\Lambda(q(x),F_n)=\Lambda(\tilde q(x),F_n)$ 
\end{itemize}
are satisfied.
Then $q(x)=\tilde q(x)$ almost everywhere in $(0,\pi).$ In particular for $n=1,$ condition (1) is equivalent to  ``$\pi/a_1\notin \mathbb Q.$"
\end{theorem}
\begin{proof}
The proof follows  the same idea as that in the proof of Theorem \ref{T1} except the computation is slightly complicated. Since
$$S(x,\lambda)=\dfrac{\sin \rho x}{\rho}+\left(\sum_{i=1}^n S(a_i,\lambda)\right)\int_0^x q(t)\dfrac{\sin \rho(x-t)}{\rho}	\, dt$$
we have 
\begin{equation}\label{E16}
0=\dfrac{\sin \rho a_j}{\rho}+\sum_{i\ne j}S(a_i,\lambda)\int_0^{a_j}q(t)\dfrac{\sin\rho(a_j-t)}{\rho}\, dt+S(a_j,\lambda)\left(\int_0^{a_j} q(t)\dfrac{\sin\rho(a_j-t)}{\rho}\, dt-1\right) 
\end{equation}
and 
\begin{equation}\label{E17}
S(\pi,\lambda)=\dfrac{\sin\rho\pi}{\rho}+\left(\sum_{i=1}^n S(a_j,\lambda)\right)\int_0^\pi q(t)\dfrac{\sin \rho(\pi-t)}{\rho}\, dt.
\end{equation}
According \eqref{E16} and \eqref{E17}, the characteristic function is 
\begin{align}\label{E18}
\triangle_n(\lambda)=\det (K_n(\lambda))
\end{align}
where 
$$K_n(\lambda)=\begin{bmatrix}
	\dfrac{\sin \rho a_1}{\rho} &\int_0^{a_1} q(t)\dfrac{\sin \rho(a_1-t)}{\rho}\, dt -1 & \int_0^{a_1} q(t)\dfrac{\sin \rho(a_1-t)}{\rho}\, dt   &\hdots & \int_0^{a_1} q(t)\dfrac{\sin \rho(a_1-t)}{\rho}\, dt  \\
	\dfrac{\sin \rho a_2}{\rho} &\int_0^{a_2} q(t)\dfrac{\sin \rho(a_2-t)}{\rho}\, dt  &\int_0^{a_2} q(t)\dfrac{\sin \rho(a_2-t)}{\rho}\, dt-1 &\hdots &\int_0^{a_2} q(t)\dfrac{\sin \rho(a_2-t)}{\rho}\, dt \\
	\vdots & \vdots & \vdots  &\ddots &\vdots \\
	\dfrac{\sin \rho a_n}{\rho} &\int_0^{a_n} q(t)\dfrac{\sin \rho(a_n-t)}{\rho}\, dt  &\int_0^{a_n} q(t)\dfrac{\sin \rho(a_n-t)}{\rho}\, dt &\hdots &\int_0^{a_n} q(t)\dfrac{\sin \rho(a_n-t)}{\rho}\, dt-1\\
	\dfrac{\sin \rho \pi}{\rho} &\int_0^{\pi} q(t)\dfrac{\sin \rho(\pi-t)}{\rho}\, dt  &\int_0^{\pi} q(t)\dfrac{\sin \rho(\pi-t)}{\rho}\, dt  &\hdots   &\int_0^{\pi} q(t)\dfrac{\sin \rho(\pi-t)}{\rho}\, dt
	\end{bmatrix}$$
	\eqref{E18} can be simplified as 
\begin{equation}
\triangle_n(\lambda)=det(E_n(\lambda))
\end{equation}
where 
$$E_n(\lambda)=\begin{bmatrix}\dfrac{\sin \rho a_1}{\rho} &\int_0^{a_1}q(t)\dfrac{\sin\rho(a_1-t)}{\rho}\, dt-1 &1 &1 &1 &\hdots &1 &1 \\
	\dfrac{\sin \rho a_2}{\rho} &\int_0^{a_2}q(t)\dfrac{\sin\rho(a_2-t)}{\rho}\, dt & -1 &0 &0 &\hdots &0 &0\\
	\dfrac{\sin \rho a_3}{\rho} &\int_0^{a_3}q(t)\dfrac{\sin\rho(a_3-t)}{\rho}\, dt & 0 & -1 &0& \hdots &0 &0 \\
	\vdots &\vdots &\vdots &\vdots &\vdots &\ddots &0 &0 \\
	\dfrac{\sin \rho a_n}{\rho} &\int_0^{a_n}q(t)\dfrac{\sin\rho(a_n-t)}{\rho}\, dt &0 &0 &0   &\hdots &0 & -1\\
	\dfrac{\sin \rho \pi}{\rho} &\int_0^{\pi}q(t)\dfrac{\sin\rho(\pi-t)}{\rho}\, dt &0 &0 &0 &0 \hdots &0 &0
	\end{bmatrix}.
$$
Since $\Lambda(q(x),F_n)=\Lambda(\tilde q(x),F_n)$, the corresponding characteristic function $\triangle_n(\lambda)=\tilde \triangle_n(\lambda).$ This yields to 
$$ \triangle_n(\lambda)-\tilde \triangle_n(\lambda)=0=\det (H_n(\lambda)),$$
where 
$H_n(\lambda)$ is the matrix obtained by substituting $\hat q(x)=q(x)-\tilde q(x)$ for $q(x)$ in $E_n(\lambda).$
Thus $$\int_0^\pi \hat q(t)\sin m(\pi-t)\, dt=0$$
	 for all $m\in \mathbb N.$ This yields to 
	 $\hat q(t)=0,$  which  completes the proof. 
\end{proof}

 \section{numerical simulations}
 In this section, we shall present a numerical algorithm  for reconstruction of penitential. For simplicity, we present 
 an example for a Sturm-Liouville operator with two frozen arguments. We shall stat the algorithm first.
  
 \vskip .3cm
 
\noindent{\bf Algorithm.} For a given  spectral set $\{\lambda_n\}_{n=1}^\infty$ of some $L(q(x),F_2,0,0),$ construct the potential $q(x)$ from $\{\lambda_n\}_{n=1}^\infty,$ where $F_2$ satisfies the assumption in Theorem 2.2. 
 
 One can reconstruct $q(x)$ according to the following steps.
 \begin{itemize}
 \item[(1)]
  Construct $\triangle_2(\lambda)=\pi \prod_{n=1}^\infty \dfrac{\lambda_n-\lambda}{n^2}.$
  In practical, we can only have a limited number of eigenvalues $\{\lambda_k\}_{k=1}^n.$ We shall
  substitute  $\triangle_2(\lambda)$ with $\prod_{k=1}^n \dfrac{\lambda_k-\lambda}{k^2})\left(\dfrac{\frac{\sin\sqrt{\lambda}\pi}{\sqrt{\lambda}}}{\prod_{k=1}^n \dfrac{k^2-\lambda}{k^2}}\right).$
  \item[(2)] Apply \eqref{E13}, we have 
    $$d_n=\dfrac{2}{\pi}\int_0^\pi q(t)\sin n(\pi -t)\, dt=\dfrac{2}{\pi }\left(\dfrac{n^2\triangle(n^2)}{\sin na_1+\sin na_2} \right)$$	
  \item[(3)] $$q(t)=\sum_{n=1}^\infty d_n \sin n(\pi-t).$$
  \end{itemize}

\noindent{\bf Example 1}. Suppose we have $a_1=1,\ a_2=\sqrt{2}$ and a partial spectrum $\{\sqrt{\lambda_1}=1.98868,\; \sqrt{\lambda_2}=2,\;, \sqrt{\lambda_3}=3.06656,\; \sqrt{\lambda_4}=4,\; \sqrt{\lambda_5}=5.00142,\;\sqrt{\lambda_6}=6,\;\sqrt{\lambda_7}=6.99975,\;\sqrt{\lambda_8}=8,\; \sqrt{\lambda_9}=8.99976,\; \sqrt{\lambda_{10}}=10\}.$ For this case, we let 
$$\triangle_2(\lambda)=\pi \prod_{n=1}^{10}(\dfrac{\lambda_n-\lambda}{n^2})\cdot\prod_{n=11}^\infty(\dfrac{n^2-\lambda}{n^2}) .$$  
The readers can refer to  figure \ref{fig1}. In this case,  the convergence is very fast.
 
 \noindent{\bf Example 2.} Let $a_1=1,\ a_2=\sqrt{2}.$ A set of  partial spectrum $\{\sqrt{\lambda_3}=2.90145,\; \sqrt{\lambda_4}=4.09132,\; \sqrt{\lambda_5}=4.98819,\; \sqrt{\lambda_6}=5.98425,\; \sqrt{\lambda_7}=7.00424,\;\sqrt{\lambda_8}=7.99936,\;\sqrt{\lambda_9}=9.00732,\;\sqrt{\lambda_{10}}=9.99528,\;\sqrt{\lambda_{11}}=10.9929,\; \sqrt{\lambda_{12}}=12.0105,\; \sqrt{\lambda_{11}}=10.9929,\; \sqrt{\lambda_{13}}=12.9998,\; \sqrt{\lambda_{14}}=13.9907,\; \sqrt{\lambda_{15}}=15.006,\;\sqrt{\lambda_{16}}=16.0035,\; \sqrt{\lambda_{17}}=16.9935,\;\sqrt{\lambda_{18}}=18.0013,\; \sqrt{\lambda_{19}}=19.0032,\; \sqrt{\lambda_{20}}=19.9997,\; \sqrt{\lambda_{21}}=20.9997\}\cup\{1.99593\pm 0.4925i\}$ was obtained.  In this case, a pair of complex eigenvalues appeared . 
 In this case, we let 
 $$ \triangle_2(\lambda)=\dfrac{((1.99593+ 0.4925i)^2-x)\times((1.99593- 0.4925i)^2-x) }{4} \cdot \prod_{n=3}^{21}(\dfrac{\lambda_n-\lambda}{n^2})\cdot\prod_{n=22}^\infty(\dfrac{n^2-\lambda}{n^2}) .$$
 The readers can refer to figure\ref{fig2} for the result of simulation,  the convergence is poor and The Gibbs phenomenon appears. 
 

\begin{figure}[htbp]
\centering
\subfigure{
\begin{minipage}{7cm}
\centering
\includegraphics[width=7cm ]{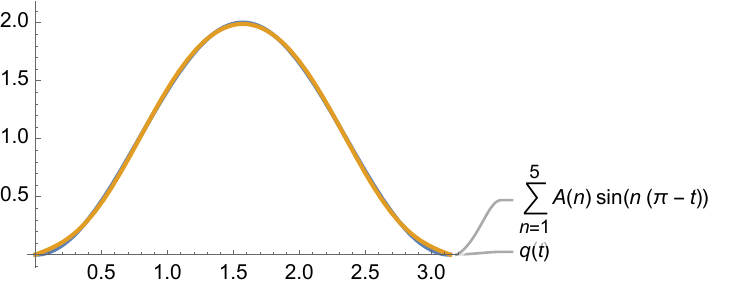}
\caption{An approximated  solution for the potential $q(t)=1-\cos(2t).$}\label{fig1}
\end{minipage}%
}\qquad 
\subfigure{
\begin{minipage}{7cm}
\centering
\includegraphics[scale=0.7]{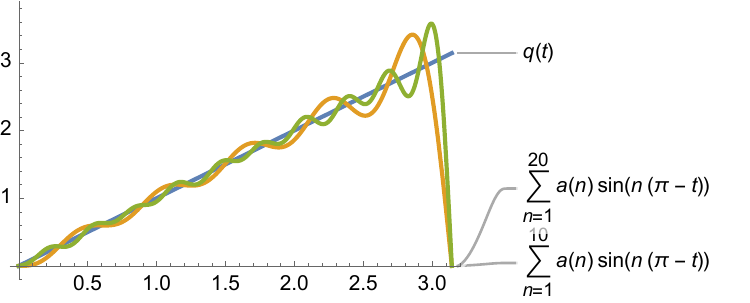}
\caption{Approximated  solutions for the potential $q(t)=t.$}\label{fig2}

\end{minipage}
}
\end{figure}

\end{document}